\documentclass[11pt,a4paper]{article}
\usepackage[margin=1in]{geometry}
\usepackage{amsmath,amssymb,amsthm,mathtools}
\usepackage{hyperref}
\hypersetup{colorlinks=true,linkcolor=blue,citecolor=blue,urlcolor=blue}
\usepackage{graphicx}
\usepackage[normalem]{ulem} 
\usepackage{pgfplots}
\pgfplotsset{compat=1.17}
\usepgfplotslibrary{fillbetween}

\newtheorem{theorem}{Theorem}[section] \newtheorem{proposition}{Proposition}[section]

\newtheorem{definition}{Definition}[section]

\newtheorem{algorithm}{Algorithm}[section]

\newcommand{\R}{\mathbb{R}}
\newcommand{\Z}{\mathbb{Z}}
\newcommand{\Q}{\mathbb{Q}}
\newcommand{\N}{\mathbb{N}}

\newcommand{\Sym}{\mathbb{S}}
\newcommand{\rank}{\mathrm{rank}\,}
\newcommand{\diag}{\mathrm{diag}\,}
\newcommand{\sos}{\Sigma}

\newcommand{\cP}{P}

\newcommand{\Span}{\mathrm{span}}

\title{Positively not SOS: pseudo-moments\\ and extreme rays in exact arithmetic}

\begin{document}
	
	\author{Didier Henrion$^{1,2}$}
	\date{\today}
	
	\footnotetext[1]{LAAS-CNRS, Universit\'e de Toulouse, France} 
	\footnotetext[2]{Faculty of Electrical Engineering, Czech Technical University in Prague, Czechia}

	\maketitle
\begin{abstract} A polynomial that is a sum of squares (SOS) of other polynomials is evidently positive. The converse is not true, there are positive polynomials which are not SOS. This note focuses on the problem of certifying, in exact arithmetic, that a given positive polynomial is not SOS. Using convex duality, this can be achieved by constructing a separating linear functional called a pseudo-moment certificate. We present  constructive procedures to compute such certificates with rational coefficients for several famous forms (homogeneous polynomials) that are known to be positive but not SOS. Our method leverages polynomial symmetries to reduce the problem size and provides explicit integer-based formulas for generating these rational certificates. As a by-product, we can also generate extreme rays of the pseudo-moment cone in exact arithmetic.

\end{abstract}

\section{Positive versus SOS cones}\label{sec:cones} 

Fix $n,d\in\mathbb{N}$. Let $H_{n,2d}$ denote the vector space of homogeneous polynomials, i.e. forms of degree $2d$ in $n$ variables. Let \[ \cP_{n,2d}\ :=\ \{\,p\in H_{n,2d},\ :\ p(x)\ge 0\ \forall x\in\R^n\,\} \] denote the set of positive (i.e. non-negative) polynomials, and let \[ \sos_{n,2d}\ :=\ \big\{\textstyle\sum_k q_k^2:\ q_k\in H_{n,d}\big\} \] denote the set of polynomial sums of squares (SOS).
Both sets are closed, full-dimensional convex cones in $H_{n,2d}$. Obviously, $\sos_{n,2d} \subset \cP_{n,2d}$, i.e. every SOS polynomial is positive. A natural question is whether the converse holds: is every positive polynomial SOS ? In his seminal 1888 work, Hilbert showed that this is only true in three specific cases ($n=2$, $2d=2$ and $n=3, 2d=4$). In all other cases, there exists a gap between 
$\cP_{n,2d}$ and $\sos_{n,2d}$. This discovery led Hilbert to pose his famous 17th problem in 1900: if a polynomial is positive, can it at least be represented as a SOS of rational functions ? In 1927, Artin provided an affirmative answer. A key consequence is that for any positive polynomial $p$, there exists a non-zero polynomial $q$ (which can itself be taken SOS) such that the product of $p$ and $q$ is SOS. This is a certificate of positivity of $p$. For background material see e.g. \cite{CLR95,R00,M08}.

This note addresses the dual question: how can we certify that a positive polynomial is \emph{not} SOS ? We leverage convex duality and exploit symmetry to construct such certificates in exact rational arithmetic. For a polynomial lying outside the SOS cone, a separating hyperplane must exist. This hyperplane is defined by a linear functional -- a \emph{pseudo-moment certificate} -- that is positive on all SOS polynomials but strictly negative on the polynomial in question. We provide a constructive method to find such rational certificates for celebrated examples of positive but not SOS forms. 

This duality is central to the Moment-SOS hierarchy, also known as the Lasserre hierarchy, a powerful framework for solving non-convex global polynomial optimization problems by creating a sequence of increasingly tight convex relaxations. This transforms the original problem into a tractable sequence of semidefinite programs. The convergence guarantees of this hierarchy are rooted in real algebraic geometry, where theorems known as Positivstellens\"atze provide certificates connecting the algebraic property of a polynomial being SOS with the geometric property of it being positive over a given set. For background material see \cite{L10,BPT12,HKL20} and more recently \cite{N23,T24}.

Our approach consists of constructing pseudo-moment certificates analytically in rational arithmetic. By exploiting the symmetries of the polynomial, we significantly reduce the number of free parameters defining the certificate. This transforms the problem into finding a rational point in a low-dimensional spectrahedral cone, a task which can often be completed by inspection. This stands in sharp contrast to numerical approaches that rely on floating-point semidefinite programming solvers and require an a posteriori rounding step to obtain an exact rational certificate. As an outcome of our symbolic construction, we can control the rank of the moment matrix, which allows us to generate extreme rays of the pseudo-moment cone in exact arithmetic.

The remainder of this note is structured as follows. Section~\ref{sec:duality} provides the necessary background on pseudo-moment certificates, their connection to moment matrices, and the role of group representation theory in simplifying the problem. In Sections~\ref{sec:motzkin-hom} through~\ref{sec:choi-lam}, we apply our constructive methodology to derive exact rational certificates for four celebrated examples of positive non-SOS forms: the Robinson and Motzkin ternary sextics, the Reznick ternary octic, and the Choi-Lam quaternary quartic. Using each of these examples, we also construct extreme rays of the pseudo-moment cone in exact arithmetic. The certificates and their analysis in Matlab can be found at
{\tt homepages.laas.fr/henrion/software/pseudomoments}

	\section{Pseudo-moment certificates}\label{sec:duality}
	
	\subsection{Definition}
	
	Let $H^*_{n,2d}$ denote the space of linear functionals on $H_{n,2d}$.
	The dual to the SOS cone is
	\[
	\sos^*_{n,2d} := \{\,\ell\in H^*_{n,2d}\ :\ \ell(p)\ge 0\ \ \forall p \in \sos_{n,2d}\,\} = \{\,\ell\in H^*_{n,2d}\ :\ \ell(q^2)\ge 0\ \ \forall q \in H_{n,d}\,\}.
	\]
	Since $\sos_{n,2d}$ is closed, the following result is a direct application of the Separating Hyperplane Theorem of convex analysis, see e.g.  \cite[Section A.4.1]{HUL01} or \cite[Section 2.5.1]{BV04}.
	\begin{theorem}\label{prop:farkas}
		Let $p\in H_{n,2d}$. Exactly one of the following two alternatives holds:
		\begin{enumerate}
			\item $p\in\sos_{n,2d}$;
			\item There exists $\ell\in\sos_{n,2d}^\ast$ such that $\ell(p)<0$.
		\end{enumerate}
	\end{theorem}
	
	\begin{definition}
	Given $p \in H_{n,2d} \setminus \Sigma_{n,2d}$, a \emph{pseudo-moment certificate} that $p$ is not SOS is a linear functional $\ell \in \Sigma^*_{n,2d}$ such that $\ell(p)<0$. Equivalently
	\[
	\ell(q^2) \geq 0 > \ell(p), \quad \forall q \in H_{n,d}.
	\]
	\end{definition}

Note that if $p \in H_{n,2d}$ is not positive, i.e. if there exists $x^* \in \R^n$ such that $p(x^*)<0$, then 
the point evaluation functional $\ell : p \mapsto p(x^*)$ is an obvious pseudo-moment certificate. A more challenging problem consists of finding a pseudo-moment certificate for a form $p \in P_{n,2d} \setminus \Sigma_{n,2d}$ which is positive but not SOS. In this case, such a certificate belongs to $\Sigma^*_{n,2d} \setminus P^*_{n,2d}$, i.e. it cannot be a convex combination of extreme points of the moment cone $P^*_{n,2d}$.

\subsection{Pseudo-moment cone}

A linear functional $\ell \in H^*_{n,2d}$ can be identified with a vector $y \in \R^{n_{2d}}$ where $n_d := \binom{n-1+d}{n-1}$, and we will use the notation $\ell_y$ to emphasize this identification. The quadratic form $\ell_y : H_{n,d} \to \R, \ q \mapsto \ell_y(q^2)$ can be identified with a matrix $M_d(y)$ which is symmetric and linear in $y$, of size $n_d$. It is called the \emph{moment matrix}. 

Positivity of the quadratic form $q \mapsto \ell_y(q^2)$ is equivalent to positive semidefiniteness of the moment matrix. Therefore we can identify the dual SOS cone with the \emph{pseudo-moment cone}
\[
\Sigma^*_{n,2d} = \{y \in \R^{n_{2d}} : M_d(y) \succeq 0\}.
\]
Since $M_d(y)$ is symmetric and linear, the pseudo-moment cone is a convex spectrahedron, a linear section of the convex cone of positive semidefinite matrices. It is described by linear matrix inequalities (LMI).

The terminology of pseudo-moment is motivated as follows. Given a positive measure on $\R^n$, the linear functional $\ell_y : p \mapsto \int p(x) d\mu(x)$ can be identified with the moment vector $y$ of $\mu$. Since $\ell_y(q^2)=\int q^2(x) d\mu(x) \geq 0$ for all $q \in H_{n,d}$, it holds $y \in \Sigma^*_{n,2d}$. There may be however vectors in $\Sigma^*_{n,2d}$ that are not moment vectors, i.e. they do not correspond to integration against a positive measure. In other words, the cone of pseudo-moments  $\Sigma^*_{n,2d}$ is an outer approximation, or relaxation of the cone of moments.

Given $p \in P_{n,2d} \setminus \Sigma_{n,2d}$, its spectrahedral cone of pseudo-moment certificates is denoted
\[
K_p := \{y \in \R^{n_{2d}} : M_d(y) \succeq 0, \ \ell_y(p) < 0 \}.
\]
Note that $K_p \subset \Sigma^*_{n,2d} \setminus P^*_{n,2d}$ since $P^*_{n,2d}$ consists of convex combinations of point evaluations (moments of Dirac measures with rank-one moment matrices).

\paragraph{Example for ternary sextics $(n=3, 2d=6)$.}
 $H_{3,3}$ is spanned by the monomials
\[
\big\{x_1^3,\ x_1^2x_2,\ x_1^2x_3,\ x_1x_2^2,\ x_1x_2x_3,\ x_1x_3^2,\ x_2^3,\ x_2^2x_3,\ x_2x_3^2,\ x_3^3\big\}
\]
arranged here by graded-lex order.
The linear functional $\ell_y\in H^*_{3,6}$ is encoded by the degree-$6$ pseudo-moments
\[
y_{\alpha}=\ell_y(x^\alpha)=\ell_y(x_1^{\alpha_1} x_2^{\alpha_2} x_3^{\alpha_3}),\ |\alpha|={\alpha_1}+{\alpha_2}+{\alpha_3}=6,
\]
so $y\in\mathbb{R}^{n_{2d}}=\mathbb{R}^{\binom{8}{2}}=\mathbb{R}^{28}$.

The moment matrix $M_3(y)$ is the $10\times 10$ symmetric matrix
\[
M_3(y)\ =\   \big[\,\ell_y(x^\alpha x^\beta)\,\big]_{|\alpha|=|\beta|=3} = \big[\,y_{\alpha +\beta}\,\big]_{|\alpha|=|\beta|=3}
\]
described entrywise as
\[
M_3(y)=\begin{bmatrix}
	y_{600} & y_{510} & y_{501} & y_{420} & y_{411} & y_{402} & y_{330} & y_{321} & y_{312} & y_{303} \\
	y_{510} & y_{420} & y_{411} & y_{330} & y_{321} & y_{312} & y_{240} & y_{231} & y_{222} & y_{213} \\
	y_{501} & y_{411} & y_{402} & y_{321} & y_{312} & y_{303} & y_{231} & y_{222} & y_{213} & y_{204} \\
	y_{420} & y_{330} & y_{321} & y_{240} & y_{231} & y_{222} & y_{150} & y_{141} & y_{132} & y_{123} \\
	y_{411} & y_{321} & y_{312} & y_{231} & y_{222} & y_{213} & y_{141} & y_{132} & y_{123} & y_{114} \\
	y_{402} & y_{312} & y_{303} & y_{222} & y_{213} & y_{204} & y_{132} & y_{123} & y_{114} & y_{105} \\
	y_{330} & y_{240} & y_{231} & y_{150} & y_{141} & y_{132} & y_{060} & y_{051} & y_{042} & y_{033} \\
	y_{321} & y_{231} & y_{222} & y_{141} & y_{132} & y_{123} & y_{051} & y_{042} & y_{033} & y_{024} \\
	y_{312} & y_{222} & y_{213} & y_{132} & y_{123} & y_{114} & y_{042} & y_{033} & y_{024} & y_{015} \\
	y_{303} & y_{213} & y_{204} & y_{123} & y_{114} & y_{105} & y_{033} & y_{024} & y_{015} & y_{006}
\end{bmatrix}.
\]
With this choice of basis
\[
\Sigma^*_{3,6}\ =\ \big\{\,y\in\mathbb{R}^{28}\ :\ M_3(y)\succeq 0\,\big\},
\]
is a spectrahedral cone defined by the single LMI \(M_3(y)\succeq 0\). Any other ordering of the degree-3 monomials produces a permutation-congruent matrix and the same cone.

\subsection{Exact certificate}

Now suppose that a form is given with integer (or rational) coefficients. We would like to restrict the search to pseudo-moment certificate with integer (or rational) coefficients.

	\begin{theorem}
		A form $p \in P_{n,2d} \setminus \Sigma_{n,2d}$  with integer (or rational) coefficients has a pseudo-moment certificate $y\in K_p$ with integer (or rational) coefficients.
	\end{theorem}
	
	\begin{proof}
		Let $y_0 \in K_p$ be a pseudo-moment certificate for $p$, i.e. $M_d(y_0) \succeq 0$ and $\ell_{y_0}(p) < 0$.
		Define the linear functional $\ell_{y_{\rm int}}$ by integration over a measure with strictly positive density, e.g., the uniform surface measure $\sigma$ on the unit sphere. Then
		$\ell_{y_{\rm int}}(q^2)\ :=\ \int_{S^{n-1}} q^2(x)\,d\sigma(x) > 0$ for every nonzero $q\in H_{n,d}$, hence $M_d(y_{\rm int})\succ0$.
		For $t\in(0,1)$, set $y_t:=(1-t)y_0+t\,y_{\rm int}$. By convexity of the positive semidefinite cone,
		\(
		M_d(y_t)=(1-t)M_d(y_0)+t\,M_d(y_{\rm int})\succ0
		\)
		for all $t>0$. Moreover $\ell_{y_t}(p)=(1-t)\ell_{y_0}(p)+t\,\ell_{y_{\rm int}}(p)$. Since $\ell_{y_0}(p)<0$ and $t\mapsto \ell_{y_t}(p)$ is continuous, there exists $\tau\in(0,1)$ with $\ell_{y_{\tau}}(p)<0$ and $M_d(y_{\tau})\succ0$. Because the map $y\mapsto M_d(y)$ is linear and the positive definite cone is open, there is a radius $\delta>0$ such that $\|y-y_{\tau}\|<\delta$ implies $M_d(y)\succ0$. By continuity of $y\mapsto \ell_y(p)$, shrinking $\delta$ if needed ensures also that $\|y-y_{\tau}\|<\delta$ implies $\ell_y(p)<0$.		
		Rational vectors are dense in $\R^{n_{2d}}$, hence one can pick $y_{\mathbb Q}\in\Q^{n_{2d}}$ with $\|y_{\mathbb Q}-y_{\tau}\|<\delta$. Then $M_d(y_{\mathbb Q})\succ0$ and $\ell_{y_{\mathbb Q}}(p)<0$, i.e. $y_{\mathbb Q} \in K_p$.
		Let $X\in\N$ be a common denominator of the coordinates of $y_{\mathbb Q}$ and set $y:=X\,y_{\mathbb Q}\in\Z^{n_{2d}}$. By linearity,
		\(M_d(y)=X\,M_d(y_{\mathbb Q})\succeq0\) and
		\(\ell_y(p)=X\,\ell_{y_{\mathbb Q}}(p)<0. \)
		Thus $y \in K_p$ is an integer pseudo-moment certificate for $p$.
	\end{proof}

\subsection{Extreme rays}

The set of linear inequalities defining the cone of sums of squares $\Sigma_{n,2d}$ corresponds to the extreme rays of its dual pseudo-moment cone $\Sigma^*_{n,2d}$.
A vector $y$ spans an \emph{extreme ray} of a cone when the set $\{ty : t \geq 0\}$ is a one-dimensional face of the cone, i.e. it cannot be written as a nontrivial sum of two different directions in the cone. 

The \emph{rank} of a vector $y \in \Sigma^*_{n,2d}$ is the rank of the corresponding moment matrix $M_d(y)$. The rank of an extreme ray is a key structural invariant. The possible ranks are known completely only in a few low-dimensional cases \cite{B12,BHORS12,BS17}.

\begin{theorem}\label{ranks}
The rank $r$ of an extreme ray of $\Sigma^*_{3,2d} \setminus P^*_{3,2d}$ satisfies $r \ge 3d-2$. For $d \ge 4$, it also satisfies $r \le \binom{d+2}{2} - 4$.
The exact ranks are known for
{sextics ($2d=6$):} $r=7$,
{octics ($2d=8$):} $r=10$ or $11$,
{decics ($2d=10$):} $13 \leq r \leq 17$,
{dodecics ($2d=12$):} $r=16$ or $18 \le r \le 24$.
The rank $r$ of an extreme ray of $\Sigma^*_{4,4} \setminus P^*_{4,4}$ is $r=6$.
\end{theorem}

Given symmetric matrices $M_k \in \Sym^n$, $k=1,\ldots,m$, consider the spectrahedral cone $$K:=\{y\in\R^m:\ M(y):=\sum_{k=1}^m y_k M_k\succeq0\}.$$ 
Deciding whether a given vector is an extreme ray of $K$ is a standard linear algebra problem \cite[Section 2.3]{RG95}.

\begin{theorem}\label{extreme}
Given $y \in K$, let $U=[u_1,\ldots,u_k]\in\R^{n\times k}$ be a basis of $\ker M(y)$.
For each $k$, define the $n\times k$ block $M_kU$, and stack them columnwise into the matrix
	\[
		B\ :=\ [\,\mathrm{vec}(M_1U)\ \ \mathrm{vec}(M_2U)\ \ \cdots\ \ \mathrm{vec}(M_mU)\,]\ \in\ \R^{n k \times m}
	\]
which represents the linear map $B : \R^m \to (\R^n)^k, y \mapsto [M(y)u_1 \cdots M(y)u_k]$. 
Then $y$ spans an {extreme ray} of $K$ if and only if $\dim\ker B=1$, i.e., $\rank B=m-1$.
\end{theorem}

\subsection{Symmetry}

When a polynomial is invariant under the action of a finite group $G$, this symmetry can be leveraged using tools from representation theory. The core idea is to find a symmetry-adapted basis which block-diagonalizes the moment matrix. For background material and applications, see \cite{GP04,D10,BGSV11,RTAL13,HMMR24,M25}. 
The following results are standard.

\begin{theorem}
Let $G$ be a finite group acting linearly on $\R^n$ and let $\rho_d: G \to \mathrm{GL}(H_{n,d})$ be the induced representation on the degree $d$ forms, with isotypic decomposition
$H_{n,d} \;=\; V_1 \oplus \cdots \oplus V_k.$
A linear functional $\ell_y:H_{n,2d}\to\R$ is called $G$-invariant if $\ell_y\big(\rho_{2d}(g)\,p\big)=\ell_y(p)$ for all $g\in G$ and $p\in H_{n,2d}$.
Fix a basis $\{v_1,\dots,v_{n_d}\}$ of $H_{n,d}$ adapted to the isotypic decomposition, i.e. a union of bases of the $V_i$. Then the moment matrix $M_d(y)\ :=\ \big[\ell_y(v_i v_j)\big]_{i,j=1}^{n_d} \in \Sym^{n_d}$
is block diagonal with one block $M_{d,i}(y)$ for each isotypic component $V_i$.
\end{theorem}

\begin{proof}
	Define the symmetric bilinear form
	\[
	B:H_{n,d}\times H_{n,d}\to\R,\qquad B(p,q)\mapsto\ell_y(pq).
	\]
	Since $\ell_y$ is $G$-invariant, $B$ is $G$-invariant in the sense that
	\[
	B(\rho_d(g)p,\rho_d(g)q) = \ell_y((\rho_d(g)p)(\rho_d(g)q))
	 = \ell_y(\rho_{2d}(g)(pq))
	 = \ell_y(pq) = B(p,q),
	\]
	for all $g\in G$ and $p,q\in H_{n,d}$. Let $L:H_{n,d}\to H_{n,d}^\ast$ be the linear map defined by
	\[
	(Lp)(\cdot)\ :=\ B(p,\cdot)\,.
	\]
	Using the $G$-invariance of $B$ one checks that $L$  {commutes} with the action of $G$:
	\[
	L\,\rho_d(g)=\rho_d(g)\,L,\qquad\forall\,g\in G
	\]
	i.e. $L$ is a $G$-{equivariant} map.
	
	By complete reducibility we have the isotypic decomposition
	$H_{n,d}=\bigoplus_{i=1}^k V_i$ and, dually,
	$H^*_{n,d}=\bigoplus_{i=1}^k V_i^\ast$, where $V_i$ collects all copies of
	an irreducible $G$-module of a fixed isomorphism type. Schur’s Lemma \cite{Serre1977} implies that every $G$-invariant bilinear form $B : V_i \times V_j \to \R$ is identically zero when $i \neq j$,
	since $V_i$ and $V_j$ have no common irreducible constituents for $i\ne j$.
	Therefore $L$ maps $V_i$ into $V^*_i$ and vanishes from $V_i$ to $V^*_j$
	when $i\ne j$.
	
	Now choose a basis of $H_{n,d}$ adapted to the direct sum
	$H_{n,d}=\bigoplus_{i=1}^k V_i$. With respect to this basis, the matrix of
	$B$ -- which is precisely the moment matrix $M_d(y)$ -- has no
	cross-terms between distinct $V_i$ and $V_j$ and hence is block diagonal, with
	one block $M_{d,i}(y)$ per isotypic component $V_i$.  
\end{proof}

\begin{theorem}\label{thm:Ginv-certificate}
If $p\in H_{n,2d} \setminus \Sigma_{n,2d}$ is $G$-invariant, then it admits a $G$-invariant pseudo-moment certificate.
\end{theorem}

\begin{proof}
Let $\ell \in K_p$ be a pseudo-moment certificate for $p$, not necessarily $G$-invariant.
Recall that $\rho_{2d}:G\to\mathrm{GL}(H_{n,2d})$ denote the representation of $G$ induced on $H_{n,2d}$, and define the dual action of $G$ on $H^*_{n,2d}$ by
\[
(g\cdot \ell)(r)\ :=\ \ell(\rho_{2d}(g^{-1})\,r\big),\qquad r\in H_{n,2d}.
	\]
We claim that $g\cdot\ell\in\Sigma_{n,2d}^\ast$ whenever $\ell\in\Sigma_{n,2d}^\ast$.
Indeed, the SOS cone is $G$-invariant:
if $r=\sum_k q_k^2\in\Sigma_{n,2d}$ then, writing $\rho_d$ for the induced action on $H_{n,d}$, it holds $\rho_{2d}(g^{-1})\,r\ =\ \sum_k (\rho_d(g^{-1})\,q_k)^2\ \in\ \Sigma_{n,2d}$.
Hence for any $r\in\Sigma_{n,2d}$,
$(g\cdot\ell)(s)=\ell(\rho_{2d}(g^{-1})\,s)\ge0,$
so $g\cdot\ell\in\Sigma^*_{n,2d}$. It follows that the dual cone $\Sigma^*_{n,2d}$ is $G$-invariant.
Define the Reynolds average of $\ell$:
\[
\bar\ell\ :=\ \frac{1}{|G|}\sum_{g\in G} g\cdot \ell\ \in\ H^*_{n,2d}.
	\]
By convexity and $G$-invariance of $\Sigma^*_{n,2d}$, we have $\bar\ell\in\Sigma^*_{n,2d}$.
Moreover, since $p$ is $G$-invariant, $\rho_{2d}(g^{-1})\,p=p$ for all $g\in G$, and therefore
\[
\bar\ell(p)\ =\ \frac{1}{|G|}\sum_{g\in G} (g\cdot\ell)(p)
\ =\ \frac{1}{|G|}\sum_{g\in G} \ell (\rho_{2d}(g^{-1})\,p)
\ =\ \frac{1}{|G|}\sum_{g\in G} \ell(p)
\ =\ \ell(p)\ <\ 0.
	\]
Thus $\bar\ell$ is $G$-invariant (by construction), belongs to $\Sigma_{n,2d}^\ast$, and separates $p$ strictly from $\Sigma_{n,2d}$.
\end{proof}

\section{Motzkin's ternary sextic}\label{sec:motzkin-hom}

The Motzkin form is
\[
p_M(x_1,x_2,x_3)
= x_1^4x_2^2 + x_1^2x_2^4 + x_3^6 \;-\; 3\,x_1^2x_2^2x_3^2.
\]
It stands as the archetypal example of a form in $P_{3,6} \setminus \Sigma_{3,6}$, see \cite{R78,R00}. Rational pseudo-moment certificates for $p_M$ were reported e.g. in \cite[Section 4]{H11} or \cite[Example 2.7.3]{N23} by rounding floating point approximations obtained by semidefinite solvers. In the sequel we show how alternative, and significantly simpler, rational and integer pseudo-moment certificates can be constructed analytically.

\subsection{Symmetry}

The Motzkin form $p_M$ is invariant under swap $(x_1,x_2,x_3)\mapsto(x_2,x_1,x_3)$ and independent sign flips $(x_1,x_2,x_3)\mapsto(-x_1,x_2,x_3)$,  $(x_1,x_2,x_3)\mapsto(x_1,-x_2,x_3)$,
	$(x_1,x_2,x_3)\mapsto(x_1,x_2,-x_3)$.
The corresponding group $\mathbb{Z}_2\times(\mathbb{Z}_2)^3$
has order $2\times 2^3 = 16$. The sign-flip parities decompose the degree-$3$ space into four invariant subspaces:
\[
\begin{aligned}
	V_1&=\Span\{x_1x_3^2,\ x_1^3,\ x_1x_2^2\}\quad(\text{odd in }x_1),\\
	V_2&=\Span\{x_2x_3^2,\ x_1^2x_2,\ x_2^3\}\quad(\text{odd in }x_2),\\
	V_3&=\Span\{x_3^3,\ x_1^2x_3,\ x_2^2x_3\}\quad(\text{odd in }x_3),\\
 V_4&=\Span\{x_1x_2x_3\}\quad(\text{odd in all three}).
\end{aligned}
\]
By parity orthogonality, the moment matrix is block diagonal in the above bases:
\[
M_3(y)\;=\;\mathrm{diag}\big(M_{31}(y),\,M_{32}(y),\,M_{33}(y),\,M_{34}(y)\big),
\]
with block sizes $3,3,3,1$ respectively.

\subsection{Orbit parameters}

By Theorem \ref{thm:Ginv-certificate}, we can assume that the pseudo-moment certificate $\ell_y$ is invariant.
Then:
\begin{enumerate}
	\item[(i)] sign flip invariance forces $y_\alpha=0$ whenever $\alpha_i$ is odd; hence $y_\alpha=0$ unless {all} coordinates of $\alpha$ are even;
	\item[(ii)]  all-even degree-$6$ triples are exactly the permutations of $(6,0,0)$, $(4,2,0)$, $(2,2,2)$;
	\item[(iii)]  swap invariance identifies the moments within each $x_1 \leftrightarrow x_2$ orbit. 
\end{enumerate}
Thus the nonzero degree-$6$ pseudo-moments are determined by 6 parameters
\[
	a:=y_{204}=y_{024},\quad
	b:=y_{402}=y_{042},\quad
	c:=y_{222},\quad
	d:=y_{600}=y_{060},\quad
	e:=y_{420}=y_{240},\quad
	f:=y_{006}.	
\]
Let us denote by $O_M : \R^6 \to \R^{28}$ the linear map that allows to construct the pseudo-moment vector $y$ from the orbit parameters $(a,b,c,d,e,f)$.

\subsection{Certificate spectrahedron}

In the above bases the blocks read
\[
M_{31}(y)=\begin{bmatrix}
	a & b & c\\[2pt]
	b & d & e\\[2pt]
	c & e & e
\end{bmatrix},\quad
M_{32}(y)=\begin{bmatrix}
	a & c & b\\[2pt]
	c & e & e\\[2pt]
	b & e & d
\end{bmatrix},\quad
M_{33}(y)=\begin{bmatrix}
	f & a & a\\[2pt]
	a & b & c\\[2pt]
	a & c & b
\end{bmatrix},\quad
M_{34}(y)=\,c.
\]
Thus the full $10\times10$ moment matrix is
\[
M_3(y)\ =\ \mathrm{diag}\big(M_{31}(y),\,M_{32}(y),\,M_{33}(y),\,M_{34}(y)\big).
\]
Note that $M_{32}(y)$ is orthogonally similar to $M_{31}(y)$. Finally, evaluating the Motzkin form gives
\[
\ell_y(p_M)\ =\ f\ -\ 3c\ +\ 2e.
\]

Note that all the parameters appear along the diagonal of $M_3(y)$, so they are all non-negative.

\begin{proposition}
The set of pseudo-moment certificates of $p_M$ is  the convex spectrahedral cone
\[
K_M  := O_M(\{(a,b,c,d,e,f) \in  \R^6_+ :\ 
f-3c+2e < 0,\ M_{31}(y) \succeq 0,\ 
M_{33}(y) \succeq 0
\}).
\]
\end{proposition}

\subsection{Exact certificates}

Let us construct rational points in $K_M$.

\begin{algorithm}\label{alg:M}

\emph{Step 0.} Choose any rational $f > 0$.

\emph{Step 1.} Choose any rationals $c\ \ge\ \tfrac{f+3}{3}$ and $\tfrac{3c - f - 1}{2} \geq e > 0$.

\emph{Step 2.} Choose any rational $a \geq \tfrac{c^2+1}{e}.$

\emph{Step 3.} Choose any rational $b \geq \max(c,\ \tfrac{2a^2}{f} - c)$.
	
\emph{Step 4.} Choose any rational $d \geq \max(\tfrac{b^2}{a},\   e+ \tfrac{e(b-c)^2}{ae-c^2})$.

\end{algorithm}

\begin{proposition}\label{prop:psd-proof-freeF}
	Algorithm \ref{alg:M} generates a rational point in $K_M$.
\end{proposition}

\begin{proof}
Step 1 ensures that 
$\ell_y(p_M) = f - 3c + 2e \le -1$.
Step 2 yields $ae - c^2\ \ge\ 1$.
Since $f>0$, the Schur complement of $M_{33}$ at the $(1,1)$ entry gives the $2\times2$ matrix $\begin{psmallmatrix}b-a^2/f & c-a^2/f\\ c-a^2/f & b-a^2/f\end{psmallmatrix}$, which is positive semidefinite if and only if $b\ge0$ and $|c-a^2/f|\le b-a^2/f$, i.e., $b\ge\max(c,\,2a^2/f-c)$. This is enforced by Step 3 and hence $M_{33}(y)\succeq0$. 
The principal minors of $M_{31}(y)$ satisfy
$ae-c^2\geq 1$, $ad-b^2\ge0$, $e(d-e)\ge0$, by Steps 2 and 4. 
The full determinant factors as
$\det M_{31}(y)=(ae-c^2)\Big(d-\big[e+\tfrac{e(b-c)^2}{ae-c^2}\big]\Big)$, which is positive by the second term in the definition of $d$ in Step 4.  Hence $M_{31}(y)\succeq0$, and by permutation similarity $M_{32}(y)\succeq0$ as well.
\end{proof}

\subsection{Extreme rays}

Let us now explain how we can control the rank of certificates.

\begin{proposition}\label{lem:rank7}
	If $y \in K_M$ then  \(\rank M_3(y) \in \{7,8,9,10\}\).
\end{proposition}

\begin{proof}
	From $f-3c+2e<0$ and $f,e\ge0$ we must have $c>0$, hence $\rank M_{34}=1$.
	The principal $2\times2$ minor of $M_{31}$ on indices $\{1,3\}$ is
	$ae-c^2\ge0$, so with $c>0$ we get $a>0$ and $e>0$. Next, the principal minor of
	$M_{33}$ on indices $\{1,2\}$ reads $fb-a^2\ge0$. Since $a>0$, this forces
	$f>0$. We will freely divide by $a$ and $f$ below.
	
	Now let us prove that $M_{33}$ cannot have rank $\le 1$. Assume for contradiction that $\rank M_{33}\le1$. Because $M_{33}\neq0$ (it has
	$a>0$ or $f>0$), we must have $\rank M_{33}=1$. This is equivalent to the vanishing of its $2\times2$ principal minors on
	$\{1,2\}$ and $\{1,3\}$, namely $fb-a^2=0$ and $fc-a^2=0$, whence $b=c={a^2}/{f}$.
	Now the principal minor of $M_{31}$ on $\{1,3\}$ gives $ae-c^2\ge0$, so we get $e\ \ge\ c^2/a\ =\ a^3/f^2$.
	Therefore $\ell_y(p_M)=f-3c+2e\ \ge\ f-{3a^2}/{f}+{2a^3}/{f^2}
	= a(1/t-3t+2t^2)$ with $t:=a/f>0$.
	Since $1/t-3t+2t^2=(t-1)^2(2t+1)/t \ge 0$,
	we obtain $\ell_y(p_M)\ge0$, contradicting $\ell_y(p_M)<0$. Hence $ \rank M_{33}\ \ge\ 2$.
	
Now let us prove that  $M_{31}$ (and $M_{32}$) cannot have rank $\le 1$.
	Assume $\rank M_{31}\le1$. As above, since $M_{31}\neq0$ (it has $c>0$ on an
	off-diagonal and nonnegative diagonal), we must have $\rank M_{31}=1$. This forces the three principal minors on $\{1,2\}$, $\{1,3\}$, and
	$\{2,3\}$ to vanish:
	$ad-b^2=0$, $ae-c^2=0$, $e(d-e)=0$.
	Because $e>0$ (Step~1), the last equality gives $d=e$. With $a,b,c\ge0$ we then
	get $b^2=ad=ae=c^2$, hence $b=c$ and $e=c^2/a$. The Schur complement condition
	for $M_{33}\succeq0$ with $f>0$ says
	$b\ \ge\ \max(c, 2a^2/f-c)$.
	Plugging $b=c$ yields $c\ge {2a^2}/{f}-c$, i.e.
	$f\ \ge\ {a^2}/{c}$.
	Using $e=c^2/a$ and this bound,
	$\ell_y(p_M)=f-3c+2e\ \ge\ {a^2}/{c}-3c+{2c^2}/{a}
	={a^3-3ac^2+2c^3}/{(ac)}
	={(a-c)^2(a+2c)}/{(ac)}\ \ge\ 0$,
	again contradicting $\ell_y(p_M)<0$. Hence $\rank M_{31}\ge2$. By the symmetry
	$(x_1\leftrightarrow x_2)$ built into the parametrization (which interchanges
	the roles of $b$ and $c$ and swaps $M_{31}$ with $M_{32}$), the same argument
	gives $\rank M_{32}\ \ge\ 2$.

	Overall we obtain
	\[
	\rank M_3
	=\rank M_{31}+\rank M_{32}+\rank M_{33}+\rank M_{34}
	\ \ge\ 2+2+2+1\ =\ 7.
	\]
	The upper bound is $\rank M_3(y)\le 3+3+3+1=10$ because each block has size at
	most $3$ (or $1$). Therefore
	$\rank M_3(y)\in\{7,8,9,10\}$.
	Finally, the four concrete parameter choices listed in the table preceding the
	proposition attain ranks $10,9,8,7$, respectively, showing that all values in
	this set actually occur on $K_M$.
\end{proof}

We now give a simple construction that always produces a
{rational} pseudo–moment certificate \(y\in K_M\) with
\(\rank M_3(y)=7\). The idea is to make the three nontrivial blocks
\(M_{31},M_{32},M_{33}\) each have rank \(2\), while
\(M_{34}=c>0\) contributes one more rank, for a total \(2+2+2+1=7\).

\begin{algorithm}\label{alg:rank7}
\emph{Step 0.} Choose any rationals $f>0$, $c>0$, $e>0$ such that
$f-3c+2e\ <\ 0$.

\emph{Step 1.} Pick a rational \(a>0\) such that
$ae-c^2\ >\ 0$ and $a^2/f > c$.

\emph{Step 2.}
Let $b := {2a^2}/{f}-c\ =\ 2x-c$.

\emph{Step 3.}
Let $d := e + (e(b-c)^2)/(ae-c^2).$
\end{algorithm}

\begin{proposition}
	Algorithm~\ref{alg:rank7} produces a rational extreme ray in $K_M.$
\end{proposition}

\begin{proof}
By construction \(f-3c+2e<0\), and all parameters are nonnegative. In particular
	\(c>0\), so \(M_{34}=c\) contributes one rank:
	\(
	\rank M_{34}=1.
	\)
	
	With \(f>0\), the Schur complement of \(M_{33}\) at the \((1,1)\) entry is
	\[
	S\ =\ \begin{bmatrix}
		b-\frac{a^2}{f} & c-\frac{a^2}{f}\\[2pt]
		c-\frac{a^2}{f} & b-\frac{a^2}{f}
	\end{bmatrix}.
	\]
	Let \(x:=a^2/f\) and \(\delta:=x-c>0\). Step 2 gives \(b-x=x-c=\delta\).
	Therefore \(S=\delta\begin{bmatrix}1&-1\\-1&1\end{bmatrix}\succeq0\) with \(\rank S=1\).
	Because \(f>0\), we conclude \(M_{33}\succeq0\) and
	\(
	\rank M_{33}=1+\rank S=2.
	\)
	Moreover, \(M_{33}\) is not rank~1 since \((b-x,c-x)\neq(0,0)\) (we imposed \(\delta>0\)).
	
	Recall
	\[
	M_{31}=\begin{bmatrix} a & b & c\\ b & d & e\\ c & e & e\end{bmatrix},\qquad
	\det M_{31}=(ae-c^2)\Bigl(d-\bigl[e+\tfrac{e(b-c)^2}{ae-c^2}\bigr]\Bigr).
	\] 
	By Step 1, \(w:=ae-c^2>0\).
	Step 3 makes the second factor equal to zero, hence
	\(\det M_{31}=0\). We check the principal \(2\times2\) minors:
	$ae-c^2\ =\ w\ >\ 0$, 
	$e(d-e)\ =\ e (e(b-c)^2)/w > 0$ since $b\ne c$.
	It remains to verify \(ad-b^2\ge 0\). Using Step 2, set \(x=a^2/f\) and \(\delta=x-c>0\), so \(b-c=2\delta\) and \(b+c=2x\). A short algebraic calculation gives
	$(ad-b^2)w = (w-2c\delta)^2\ \ge\ 0$.
	Hence \(ad-b^2\ge0\). Since at least two principal minors are strictly positive,
	\(M_{31} \succeq 0\) and \(\rank M_{31}=2\) (it cannot drop to rank~1 because that
	would force all three principal minors of order \(2\) to vanish, contradicting
	\(e(d-e)>0\)). We also have \(\rank M_{32}=2\).
	
	Summing the block ranks
	\[
	\rank M_3 =\rank M_{31}+\rank M_{32}+\rank M_{33}+\rank M_{34}
	=2+2+2+1=7
	\]
	completes the proof.
\end{proof}

As an illustration, with the initial triple $(c,e,f)=(2,2,1)$ (so that $f-3c+2e=1-6+4=-1<0$)
and $a=3$, the choices below realize all possible ranks; each line differs from the previous by saturating exactly one more boundary equality.

\begin{center}
	\renewcommand{\arraystretch}{1.12}
	\begin{tabular}{r|cccccc}
		$\rank M_3$ & $a$ & $b$ & $c$ & $d$ & $e$ & $f$ \\
		\hline
		$10$ & $3$ & $17$ & $2$ & $228$ & $2$ & $1$ \\
		$9$  & $3$ & $16$ & $2$ & $199$ & $2$ & $1$ \\
		$8$  & $3$ & $17$ & $2$ & $227$ & $2$ & $1$ \\
		$7$  & $3$ & $16$ & $2$ & $198$ & $2$ & $1$ \\
	\end{tabular}
\end{center}

For the pseudo-moment vector $y = O_M(3,16,2,198,2,1) \in \N^{28}$, the corresponding $10\times10$ moment matrix $M_3(y)$ has rank 7. Vector $y$ must be an extreme ray of $\Sigma^*_{3,6}$, see Theorem \ref{ranks}, which also shows that lower values of the rank are not possible in $\Sigma^*_{3,6} \setminus P_{3,6}$. 

As a sanity check, we can use Theorem~\ref{extreme} and construct a kernel basis $U \in \Z^{10\times 3}$ consistent with the parity blocks:
\[
\small
U=
\begin{bmatrix}
	-14&  0&  0\\
	1&  0&  0\\
	13&  0&  0\\ \hline
	0&-14&  0\\
	0& 13&  0\\
	0&  1&  0\\ \hline
	0&  0& -6\\
	0&  0&  1\\
	0&  0&  1\\ \hline
	0&  0&  0
\end{bmatrix}
\]
and satisfying $M_3(y)\,U=0$. Then we construct the matrix \(B\in\mathbb Z^{30\times 28}\) and using fraction-free Gaussian elimination we obtain \(\rank B=27\), showing extremality of $y$.

Other integer extreme rays can be generated with
$f=1$, $c=2$, $e=2$ and $a \in \Z$, $a \geq 3$. Let $b=2a^2-2$ and $d = 2+(2a^2-4)^2/(a-2)$. This last quantity is integer if and only if $a-2$ divides 16, or equivalently $a \in \{3,4,6,10,18\}$.
Then $M_{33}(y)\succeq0$ and $\det M_{33}(y)=0$ (rank $2$) by the Schur–complement equality
$b-a^2/f = |c-a^2/f|$ equivalent to $b = 2a^2-c=2a^2-2$.
Also $M_{31}(y)\succeq0$ and $\det M_{31}(y)=0$ (rank $2$) because $ae-c^2=2(a-2)>0$ and
$d =e+(e(b-c)^2)/(ae-c^2) = 2+(2a^2-4)^2/(a-2)$
is precisely the determinant–zero choice; moreover $ad-b^2>0$ holds for $a\ge 3$. Finally, the separation is strict: $f-3c+2e = 1-6+4=-1<0$. Hence both $3\times3$ blocks have rank $2$ and all other invariant diagonal entries of $M_3(y)$ are positive; therefore $\rank M_3(y)=7$.

The following table provides the corresponding extreme rays of $\Sigma^*_{3,6}$:
\[
\begin{array}{cccccc}
	a & b & c & d & e & f\\\hline
	3 & 16 & 2 & 198   & 2 & 1\\
	4 & 30 & 2 & 394   & 2 & 1\\
	6 & 70 & 2 & 1158  & 2 & 1\\
	10& 198& 2 & 4804  & 2 & 1\\
	18& 646& 2 & 25923 & 2 & 1
\end{array}
\]

\section{Robinson's ternary sextic}\label{sec:robinson}

\subsection{Symmetry}

The Robinson form
\[
p_R(x_1,x_2,x_3)
= x^6_1 + x^6_2 + x^6_3 - (x^4_1 x^2_2 + x^2_1 x^4_2 + x^4_1 x^2_3 + x^2_1 x^4_3 + x^4_2 x^2_3 + x^2_2 x^4_3) + 3 x^2_1 x^2_2 x^2_3
\]
is another well-studied member of $P_{3,6} \setminus \Sigma_{3,6}$, see \cite{R78,R00}.
It is invariant under the group  $B3=S_3\times(\mathbb{Z}_2)^3$ acting on polynomials by permuting variables and flipping signs. This group is known as the hyperoctahedral group and it has order $3!\times2^3=48$.
Parity under sign flips decomposes the degree-$3$ space into the same 4 subspaces as for the Motzkin form. Therefore the moment matrix $M_3(y)$ is block diagonal with three $3\times 3$ blocks and one $1\times1$ block. Thanks to the action of the full permutation group $S_3$, the three $3\times 3$ blocks are identical up to row and column permutations.

\subsection{Orbit parameters}

Invariance and homogeneity force all degree-$6$ moments to depend only on the $S_3$-orbit type. Let
\[
a:=y_{600}=y_{060}=y_{006},\qquad
b:=y_{420}=y_{402}=y_{240}=y_{204} =y_{042}=y_{024},\qquad
c:=y_{222}.
\]
Let us denote by $O_R : \R^3 \to \R^{28}$ the linear map that allows to construct the pseudo-moment vector $y$ from the orbit parameters $(a,b,c)$.

\subsection{Certificate spectrahedron}

In the ordered basis $\{x_1^3,\ x_1x_2^2,\ x_1x_3^2\}$ (and analogously for the other two copies), the $3\times3$ block reads
\[
M_{31}(y)\ =\
\begin{bmatrix}
	a & b & b\\
	b & b & c\\
	b & c & b
\end{bmatrix}.
\]
Thus
\[
M_3(y)\ =\ \mathrm{diag}(M_{31}(y),\ 	M_{31}(y),\ 	M_{31}(y),\ c).
\]

Note that all the parameters appear along the diagonal of $M_3(y)$, so they are all non-negative. The $2\times2$ principal minors of $M_{31}(y)$ give
\[
a \geq b \geq c.
\]
A direct expansion shows
\[	\det M_{31}(y)\ =\ (b-c)\,\big(a\,b+a\,c-2b^2\big),
\]
so, combined with $b\ge c$, the $3\times3$ positivity reduces to
\[
a(b+c)\ \ge\ 2b^2.
\]
Evaluation of the Robinson form yields
\[
\ell_y\big(p_R\big)\ =\ 3a\ -\ 6b\ +\ 3c\ =\ 3\,(a-2b+c).
\]

\begin{proposition}
The set of pseudo-moment certificates of $p_R$ is  the convex quadratic cone
\[
K_R := O_R(\{(a,b,c) \in  \R^3_+: a-2b+c < 0,\  
a \geq b \geq c,\ a(b+c)\ge 2b^2 \})
\]
\end{proposition}

See Figure \ref{fig:kr} for a representation of a compact section of $K_R$.

\begin{figure}[h!]
	\centering
	\begin{tikzpicture}
		\begin{axis}[
			axis lines=middle,
			xlabel={$b$},
			ylabel={$c$},
			title={},
			xmin=0, xmax=1,
			ymin=0, ymax=1,
			enlarge x limits=false,
			enlarge y limits=false,
			axis equal image,
			grid=major,
			grid style={dashed, gray!30},
			]
			
			\addplot[
			domain=0:1,
			samples=2,
			color=black,
			ultra thick,
			name path=upper
			] {x};
			
			\addplot[
			domain=0.5:1,
			samples=100,
			color=black,
			ultra thick,
			name path=lower_parabola
			] {2*x^2 - x};
			
			\draw[black, ultra thick, name path=lower_line] (axis cs:0,0) -- (axis cs:0.5,0);
			
			\addplot[gray!40] fill between[
			of=upper and lower_parabola,
			soft clip={domain=0.5:1}
			];
			\addplot[gray!40] fill between[
			of=upper and lower_line,
			soft clip={domain=0:0.5}
			];
			
		\end{axis}
	\end{tikzpicture}
	\caption{Cross section $a=1$ in the orbit plane $(b,c)$ of pseudo-moment certificates for the Robinson form.}
	\label{fig:kr}
\end{figure}
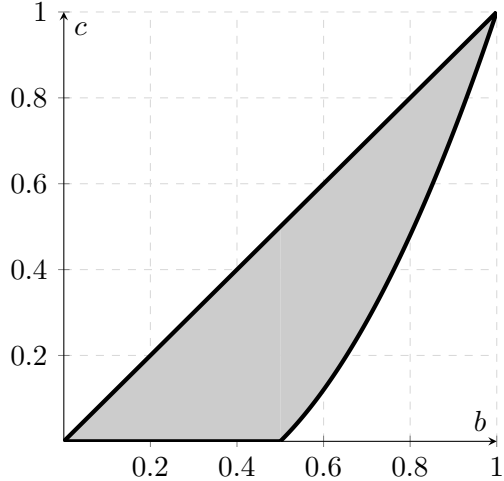

\subsection{Exact certificates}

A simple rational point $y = O_R(a,b,c) \in K_R$ is obtained with the orbit parameters
\[
a=1,\qquad b=\tfrac{2}{3},\qquad c=\tfrac{1}{4}.
\]
Then $a \geq b \geq c$, $a(b+c)=\tfrac{11}{12}\ge 2b^2=\tfrac{8}{9}$, so $M_{31}(y)\succeq0$. Moreover $ 1-2\cdot\tfrac23+\tfrac14 =  -\tfrac{1}{12}\ <\ 0.$
For this choice $\rank M_3(y)=10$, so the certificate is in the interior of $\Sigma^*_{3,6} \setminus P^*_{3,6}$.

If we enforce the mass $a$ to one, then there is no integer pseudo-moment certificate. This is apparent on Figure \ref{fig:kr}. More formally, if $y \in K_R$ with \(a=1\) then \(b\in\{0,1\}\) and \(c\le b\). If \(b=0\) then necessarily \(c=0\). The negativity condition reads \(1-0+0<0\), i.e.\ \(1<0\), impossible. If \(b=1\) then \(c\in\{0,1\}\) and the quadratic constraint \(b+c\ge 2b^2\) gives \(1+c\ge2\) and hence $c=1$.
The negativity condition becomes \(1-2+1=0<0\), impossible. 

Also apparent from Figure \ref{fig:kr} is that a sufficiently large integer cross section $a > 1$ will generate integer points in $K_R$. Denoting by $N(a)$ the number of integer points, it can be checked that $N(1) = \cdots = N(7) = 0$, $N(8)=1$, $N(9)=N(10)=2$, $N(11) = 3$ and $N(a) = a^2/24 + O(a)$.

Let us now try to find integer certificates of minimal size:

\emph{Step 1.}
For $b=1$: $a(b{+}c)\ge2$ forces $(c,a)=(0,\ge2)$ or $(1,\ge1)$, but $a-2b+c<0$ gives $a{+}c<2$, impossible.
For $b=2$: $a\ge\lceil8/(2{+}c)\rceil$ while $a-4+c<0$; checking $c=0,1,2$ shows no integer $a$ satisfies both.
For $b=3$: $a\ge\lceil18/(3{+}c)\rceil$ and $a-6+c<0$; for $c=0,1,2,3$ each case contradicts $a\ge\lceil18/(3{+}c)\rceil$.
For $b=4$: $a\ge\lceil32/(4{+}c)\rceil$ and $a-8+c<0$; $c=0,1,2,3,4$ are all infeasible.
For $b=5$: $a\ge\lceil50/(5{+}c)\rceil$ and $a-10+c<0$; $c=0,\dots,5$ are all infeasible.

\emph{Step 2.}
For $b=6$ one has $a\ge\lceil72/(6{+}c)\rceil$ and $a-12+c<0$. The feasible integer solutions are
$(c,a)=(2,9)\quad\text{and}\quad(c,a)=(3,8)$,
yielding the two minimal triples $(a,b,c)=(9,6,2)\quad\text{and}\quad(8,6,3)$,
both with $a{+}b{+}c=17$.

\emph{Step 3.}
If $b\ge7$, then $a\ge b$ and $c\le b$ imply $a{+}b{+}c\ge 3b\ge21>17$.

Therefore the minimal possible integer sum is
$a{+}b{+}c=17$,
achieved exactly by the two triples $(9,6,2)$ and $(8,6,3)$, which satisfy all feasibility conditions and yield $\ell_y(p_R)=3(a-2b+c)=-3<0$ in both cases.

\subsection{Extreme rays}

Extreme rays of $\Sigma^*_{3,6} \setminus P^*_{3,6}$ have rank 7, see Theorem \ref{ranks}.
To generate such a ray, enforce $\det M_{31}(y) = 0$ i.e. $a(b+c)=2b^2$, along the parabolic boundary on Figure \ref{fig:kr}. It can be checked that the number of integer extreme rays grows in $O(\sqrt{a})$. The smallest of them are $(a,b,c) \in \{(8,6,3),(9,6,2),(16,12,6),(18,12,4),(24,18,9)\}.$


For the minimum size integer point  \((a,b,c)=(8,6,3)\), the moment matrix block
\[
M_{31}=\begin{bmatrix}
	8 & 6 & 6\\
6 & 6 & 3\\
6 & 3 & 6
\end{bmatrix}
\]
has eigenvalues $\{\,0,\ 3,\ 17\,\}$.
When \((a,b,c)=(9,6,2)\) the moment matrix block
\[
M_{31}=\begin{bmatrix}
	9 & 6 & 6\\
6 & 6 & 2\\
6 & 2 & 6
\end{bmatrix}
\]
has eigenvalues $\{\,0,\ 4,\ 17\,\}$.
These are integer extreme rays since 
$\rank M_3(y) = 2+2+2+1= 7$.

\section{Reznick's Ternary Octic}

Consider the ternary octic
\[
p_8(x_1,x_2,x_3)=
x_{1}^{2}x_{3}^{6} + x_{2}^{2}x_{3}^{6} + x_{1}^{4}x_{2}^{4} - 3x_{1}^{2}x_{2}^{2}x_{3}^{4}
\]
described by Reznick in \cite[Section 7, case $m=4$]{R78} as a member of $P_{3,8} \setminus \Sigma_{3,8}$.

\subsection{Symmetry}\label{subsec:octic-sym}

The Reznick form $p_8$ is invariant under the same group  as the Motzkin form. The degree-4 space decomposes into four invariant subspaces:
\[
\begin{aligned}
	V_1&=\Span\{x_1^3x_2,\;x_1x_2^3,\;x_1x_2x_3^2\}\quad(\text{odd in }x_1,x_2),\\
	V_2&=\Span\{x_1^3x_3,\;x_1x_3^3,\;x_1x_2^2x_3\}\quad(\text{odd in }x_1,x_3),\\
	V_3&=\Span\{x_2^3x_3,\;x_2x_3^3,\;x_1^2x_2x_3\}\quad(\text{odd in }x_2,x_3),\\
	V_4&=\Span\{x_1^4,\;x_2^4,\;x_3^4,\;x_1^2x_2^2,\;x_1^2x_3^2,\;x_2^2x_3^2\}\quad(\text{even in all}).
\end{aligned}
\]
Hence the homogeneous moment matrix of degree $4$
is block diagonal in the above bases, i.e.
\[
M_4(y)\ =\ \mathrm{diag}\big(M_{41}(y),\,M_{42}(y),\,M_{43}(y),\,M_{44}(y)\big),
\]
with block sizes $3,3,3,6$, respectively.

\subsection{Orbit parameters}\label{subsec:octic-orbits}

Under group invariance, the pseudo-moment certificate $\ell_y$ is determined by $9$ parameters
\[
\begin{gathered}
	a:=y_{800}=y_{080},\quad b:=y_{008}, \quad
	c:=y_{620}=y_{260},\quad d:=y_{602}=y_{062},\quad e:=y_{206}=y_{026},\\
	f:=y_{440},\quad g:=y_{404}=y_{044},\quad
	h:=y_{422}=y_{242},\quad i:=y_{224}.
\end{gathered}
\]
Let $O_8:\R^9\to\R^{N_8}$ denote the linear map that assigns to
$(a,b,c,d,e,f,g,h,i)$ the full degree-$8$ pseudo-moment vector by replicating
these values on their orbits.

\subsection{Certificate spectrahedron}\label{subsec:octic-spectrahedron}

In the bases above, the blocks of $M_4(y)$ are
\[
M_{41}(y)=\begin{bmatrix}
	c & f & h\\[2pt]
	f & c & h\\[2pt]
	h & h & i
\end{bmatrix},\qquad
M_{42}(y)=\begin{bmatrix}
	d & g & h\\[2pt]
	g & e & i\\[2pt]
	h & i & h
\end{bmatrix},\qquad
M_{43}(y)=M_{42}(y),
\]
and the $6\times6$ even-parity block is
\[
M_{44}(y)=\begin{bmatrix}
	a & f & g & c & d & h\\
	f & a & g & c & h & d\\
	g & g & b & i & e & e\\
	c & c & i & f & h & h\\
	d & h & e & h & g & i\\
	h & d & e & h & i & g
\end{bmatrix}.
\]
Finally, evaluating $p_8$ under $\ell_y$ uses only three orbit parameters and reads
\[
\ell_y(p_8)\;=\;f\ +\ 2\,e\ -\ 3\,i.
\]

With the orthonormal change of basis  
\[
Q_{44} = \begin{bmatrix}
	\frac{1}{\sqrt{2}} & 0 & \frac{1}{\sqrt{2}} & 0 & 0 & 0\\[4pt]
	-\frac{1}{\sqrt{2}} & 0 & \frac{1}{\sqrt{2}} & 0 & 0 & 0\\[4pt]
	0 & 0 & 0 & 1 & 0 & 0\\[4pt]
	0 & 0 & 0 & 0 & 1 & 0\\[4pt]
	0 & \frac{1}{\sqrt{2}} & 0 & 0 & 0 & \frac{1}{\sqrt{2}}\\[4pt]
	0 & -\frac{1}{\sqrt{2}} & 0 & 0 & 0 & \frac{1}{\sqrt{2}}
\end{bmatrix}
\]
we obtain
\[
Q^T_{44} M_{44}(y)\,Q_{44} \;=\;
\diag\Bigl(\underbrace{\begin{bmatrix}
		a-f & d-h\\[2pt]
		d-h & g-i
\end{bmatrix}}_{M_{441}(y)}, \quad 
\underbrace{\begin{bmatrix}
		a+f & \sqrt2\,g & \sqrt2\,c & d+h\\[2pt]
		\sqrt2\,g & b & i & \sqrt2\,e\\[2pt]
		\sqrt2\,c & i & f & \sqrt2\,h\\[2pt]
		d+h & \sqrt2\,e & \sqrt2\,h & g+i
\end{bmatrix}}_{M_{442}(y)}\Bigr).
\]
Consequently $M_{441}(y) \succeq 0$ amounts to the convex inequalities
$a+g-f-i\ge0,\ (a-f)(g-i)\ge(d-h)^2$.

With the orthonormal change of basis  
\[
Q_{41}=\begin{bmatrix}
	-\frac{1}{\sqrt2} & \frac{1}{\sqrt2} & 0\\[2pt]
	\frac{1}{\sqrt2} & \phantom{-}\frac{1}{\sqrt2} & 0\\[2pt]
	0 & 0 & 1
\end{bmatrix}
\]
we obtain
\[
Q^T_{41} M_{41}(y)\,Q_{41}
=
\diag\Bigl(c-f, \underbrace{
	\begin{bmatrix}
		c+f & \sqrt{2}\,h\\
		\sqrt{2}\,h & i
\end{bmatrix}}_{M_{412}(y)}\Bigr).
\]
and hence $M_{41}(y) \succeq 0$ amounts to the convex  inequalities $c^2 \geq f^2$, 
\((c+f)i\ge 2h^2\). 

Note that all the parameters appear along the diagonal of $M_4(y)$, so they are all non-negative. 

\begin{proposition}
	The set of pseudo-moment certificates of $p_8$ is  the convex spectrahedral cone
\[
\begin{aligned}
	&K_8
	:=\ O_8(\{(a,b,c,d,e,f,g,h,i)\in\mathbb{R}^9_{+}\ :\ f+2e-3i<0, \quad 
	c^{2}\ge f^{2},\quad (c+f)i\ge 2h^{2},\\
	&\hspace{10em} a+g\ge f+i,\quad (a-f)(g-i)\ge (d-h)^{2},\quad  M_{442}(y)\succeq 0 \}).
\end{aligned}
\]
\end{proposition}

\subsection{Exact certificates}

Let us construct rational points in \(K_8\).

\begin{algorithm}\label{alg:k8}
\emph{Step 0.} Choose any rationals \(e,f,h,i>0\) such that \(f+2e-3i \le -1\) and \(eh-i^2 > 0\).

\emph{Step 1.} Choose any rational \(g\ \ge\ \max\!\big(i,\ \frac{2h^2-if+1}{f}\big)\).

\emph{Step 2.} Choose any rational \(d \ge \max\!\big(\frac{g^2}{e},\ \frac{h(g-i)^2}{eh-i^2}+h,\ \frac{g(g+i)}{e}-h\big)\).

\emph{Step 3.} Choose any rational \(c \ge \max\!\big(\frac{h(d+h)}{g+i},\ \frac{2h^2-fi}{i},\ f\big)\).

\emph{Step 4.} Choose any rational
$b \ \ge\ \frac{2e^2}{g+i}\ +\ \frac{((g+i)i-2eh)^2}{(g+i)\,((g+i)f-2h^2)}$.

\emph{Step 5.} Choose any rational \(a \ge \max\!\big(\frac{(d-h)^2}{g-i}+f,\ \frac{(d+h)^2}{g+i}-f\big)\).
\end{algorithm}

\begin{proposition}\label{prop:octic-det-rat-final}
Algorithm \ref{alg:k8} generates a rational point in \(K_8\).
\end{proposition}

\begin{proof}
	All parameters are nonnegative by construction, and \(\ell_y(p_8)=f+2e-3i\le -1<0\) by Step~0.
	
 \emph{(i) Block \(M_{41}\).}
	Step~3 enforces \(c\ge f\) and \((c+f)i\ge 2h^2\), hence \(M_{41}\succeq0\).
	
 \emph{(ii) Block \(M_{441}\).}
	Step~1 gives \(g-i\ge0\); Step~5 gives \(a-f\ge (d-h)^2/(g-i)\), hence \(M_{441}\succeq0\).
	
 \emph{(iii) Blocks \(M_{42}\) and \(M_{43}\).}
	Their principal \(2\times2\) minors are positive by Steps~0 and~2:
\(
	de-g^2\ge 0
	\) from \(d\ge g^2/e\), and
	\(
	eh-i^2>0
	\) from Step~0.
	The full determinant admits the exact factorization
	\(
	\det M_{42}
	= (eh-i^2)\,(d-h)\ -\ h\,(g-i)^2.
	\)
	Step~2 enforces \(d-h\ \ge\ h\,(g-i)^2/(eh-i^2)\), hence \(\det M_{42}\ge 0\), so \(M_{42}\succeq0\) and likewise \(M_{43}\succeq0\).
	
 \emph{(iv) Block \(M_{442}\).}
	Taking the Schur complement w.r.t.\ \(g+i>0\) (Step~1) gives
	\[
	S=
	\begin{bmatrix}
		a+f & \sqrt2\,g & \sqrt2\,c\\
		\sqrt2\,g & b & i\\
		\sqrt2\,c & i & f
	\end{bmatrix}
	-\frac{1}{g+i}
	\begin{bmatrix}d+h\\ \sqrt2\,e\\ \sqrt2\,h\end{bmatrix}
	\begin{bmatrix}d+h & \sqrt2\,e & \sqrt2\,h\end{bmatrix}.
	\]
	Its lower \(2\times2\) block equals
	\[
	S_{23}=
	\begin{bmatrix}
		b-\frac{2e^2}{g+i} & \frac{s_2}{g+i}\\[1pt]
		\frac{s_2}{g+i} & \frac{s_3}{g+i}
	\end{bmatrix}
	\]
	with \(s_2:=(g+i)i-2eh\) and \(s_3:=(g+i)f-2h^2>0\) (Step~1), and Step~4 makes the Schur complement of \(S_{23}\) nonnegative. Hence \(S_{23}\succeq0\).
	The cross terms satisfy
	\[
	S_{12}=\sqrt2\!\left(g-\frac{e(d+h)}{g+i}\right),\quad
	S_{13}=\sqrt2\!\left(c-\frac{h(d+h)}{g+i}\right),\quad
	S_{11}=a+f-\frac{(d+h)^2}{g+i}.
	\]
	By Step~2 we may choose \(d+h=\frac{g(g+i)}{e}\) (it lies in the max), which makes \(S_{12}=0\). By Step~3 we may choose \(c=\frac{h(d+h)}{g+i}\), which makes \(S_{13}=0\). Finally Step~5 gives \(S_{11}\ge0\). Therefore \(S=\operatorname{diag}(S_{11},S_{23})\succeq0\), so \(M_{442}\succeq0\).
\end{proof}

If integer certificates are desired, replace each lower bound above by its ceiling (and keep the strict margin \(f+2e-3i\le-1\)). Monotonicity of all inequalities preserves feasibility.

With \((e,f,h,i)=(3,2,4,3)\) and the choices
	\((a,b,c,d,g)=(2392,25,40,166,14)\), one gets
	\[
	(a,b,c,d,e,f,g,h,i)=(2392,25,40,166,3,2,14,4,3).
	\]
We can check that the elementary symmetric functions of the eigenvalues of all the matrix blocks are strictly positive. Hence \(\rank M_4(y)=3+3+3+6=15\) (maximal).
	
With \((e,f,h,i)=(4,3,5,4)\) and
	\((a,b,c,d,g)=(1159,50,33,107,13)\) one gets
	\[
	(a,b,c,d,e,f,g,h,i)=(1159,50,33,107,4,3,13,5,4).
	\]
	Here \(M_{441}\) sits on the boundary:
	\((a-f)(g-i)-(d-h)^2=1156\cdot 9-102^2=0\),
	so \(M_{441}\) has one zero eigenvalue, whereas the other blocks are strictly definite.

\subsection{Extreme rays}
\label{subsec:octic-lowrank}

We now explain how to synthesize {low–rank} pseudo–moment certificates on the boundary of
\(\Sigma^*_{3,8}\), 
starting from a full rank configuration.

\begin{algorithm}\label{alg:p8e}
\emph{Step 0.} Enforce $g>i$, $eh-i^2>0$ and $\ell_y(p_8)=f+2e-3i<0$.
	
\emph{Step 1.} Let
$d=h+\frac{h(g-i)^2}{eh-i^2}$, so that $\rank M_{42} = \rank M_{43}=2$.

\emph{Step 2.} Let
$a=f+\frac{(d-h)^2}{g-i}$, so that $\rank M_{441}=1.$

\emph{Step 3.} Let
$b=\frac{2e^2}{g+i}+\frac{s_2^2}{(g+i)s_3}$, where $
s_2=(g+i)i-2eh$, $s_3=(g+i)f-2h^2>0$, so that the $2\times2$ Schur subblock $S_{23}$ of $M_{442}$ is singular.

\emph{Step 4.} Let $c \ =\ \frac{h(d+h)}{g+i}\ +\ \frac{s_3}{s_2}\!\left(g-\frac{e(d+h)}{g+i}\right)$
	which aligns $S_{12}:S_{13}=s_2:s_3$, ensuring the full Schur complement $S\succeq0$ with $\rank S=2$ and hence $\rank M_{442}=1+\rank S=3$.
\end{algorithm}

The combinations below give exactly the indicated ranks:
\[
\renewcommand{\arraystretch}{1.15}
\begin{array}{c|c|c}
	\text{equalities} & \rank(M_{41},M_{42},M_{43},M_{441},M_{442}) & \rank M_4\\\hline
	\text{none} & (3,3,3,2,4) & 15\\
	\text{Step 3 only} & (3,3,3,2,3) & {14}\\
	\text{Step 2 only} & (3,3,3,1,4) & {14}\\
	\text{Step 1 only} & (3,2,2,2,4) & {13}\\
	\text{Steps 2 \& 3} & (3,3,3,1,3) & {13}\\
	\text{Steps 1 \& 2} & (3,2,2,1,4) & {12}\\
	\text{Steps 1 \& 3} & (3,2,2,2,3) & {12}\\
	\text{Steps 1 \& 2 \& 3} & (3,2,2,1,3) & {11}
\end{array}
\]
In all cases enforce Step~4 (or choose $c$ slightly larger) to keep $M_{442}\succeq0$ and $M_{41}\succ0$, and retain
$\ell_y(p_8)=f+2e-3i<0$ from Step~0.

\begin{proposition}
		If $y \in K_8$ then  \(\rank M_4(y) \in \{10,11,12,13,14,15\}\).
\end{proposition}

\begin{proposition}	
	Algorithm \ref{alg:p8e} generates a rational extreme ray in $K_8$.
\end{proposition}

\begin{proof}
Combining the 3 steps yields $\rank M_4(y)=3+2+2+1+3=11$ and from Theorem \ref{ranks} we know that for this rank it is an extreme ray of $\Sigma^*_{3,8} \setminus P^*_{3,8}$.
\end{proof}
	
As an illustration, at Step 0 pick
$(e,f,h,i,g)=(6,2,5,5,21)$ so that $f+2e-3i=2+12-15=-1<0$, and $eh-i^2=30-25=5>0$. At Step 1, $d=5+\tfrac{5\cdot 16^2}{5}=261$. At Step 2,
$a=2+\tfrac{256^2}{16}=4098$. Then $s_3=2$ and $s_2=70$ so that $b=\tfrac{72}{26}+\tfrac{4900}{52}=97$ at Step 3.
Finally at Step 4, 
$c =\tfrac{1330}{26}+\tfrac{2}{70}(21-\tfrac{1596}{26})=50.$
The resulting
\[
(a,b,c,d,e,f,g,h,i)=(4098,\ 97,\ 50,\ 261,\ 6,\ 2,\ 21,\ 5,\ 5)
\]
is an integer rank 11 extreme ray of $\Sigma^*_{3,8}$.

In the table we report integer pseudo-moment certificates for Reznick's ternary octic.
\begin{center}
	\renewcommand{\arraystretch}{1.12}
	\begin{tabular}{r|ccccccccc}
		$\rank M_4$ & $a$ & $b$ & $c$ & $d$ & $e$ & $f$ & $g$ & $h$ & $i$\\
		\hline
		$15$  & $1194$ & $50$ & $33$ & $107$ & $4$ & $3$ & $13$ & $5$ & $4$\\
		$14$  & $1159$ & $50$ & $33$ & $107$ & $4$ & $3$ & $13$ & $5$ & $4$\\
		$13$  & $1445$ & $14$ & $40$ & $126$ & $5$ & $4$ & $15$ & $6$ & $5$\\
		$12$  & $1444$ & $14$ & $40$ & $126$ & $5$ & $4$ & $15$ & $6$ & $5$\\
		$11$  & $4098$ & $97$ & $50$ & $261$ & $6$ & $2$ & $21$ & $5$ & $5$		
	\end{tabular}
\end{center}
 
We have not been able to use this method to construct rank 10 certificates.
From Theorem \ref{ranks}, we know however that extreme rays of $\Sigma^*_{3,8}$ of rank 10 can be constructed as pseudo-moment certificates of other forms in $P_{3,8} \setminus \Sigma_{3,8}$.

\section{Choi-Lam quaternary quartic}\label{sec:choi-lam}

\subsection{Symmetry}

The Choi-Lam form
\[
p_{CL}(x_1,x_2,x_3,x_4)
= x_1^2 x_2^2 + x_2^2 x_3^2 + x_1^2 x_3^2
+ x_4^4 - 4\,x_1 x_2 x_3 x_4
\]
is a classic element of $P_{4,4} \setminus \Sigma_{4,4}$, see e.g. \cite{R78}. It is invariant under permutation of the variables $(x_1,x_2,x_3)$ by the group $S_3$. It is also invariant under sign flips of the variables, but only if an even number of signs are flipped. This sign-flip group is a subgroup of $(\Z_2)^4$ isomorphic to $(\Z_2)^3$. The full symmetry group has order $3! \times 2^3 = 48$. The degree-$2$ monomial space
decomposes into invariant blocks:
\[
V_1=\Span\{x_1^2,x_2^2,x_3^2\},\:\:
V_2=\Span\{x_1x_2,x_2x_3,x_3x_1\},\:\:
V_3=\Span\{x_1x_4,x_2x_4,x_3x_4\},\:\:
V_4=\Span\{x_4^2\}.
\]
Consequently, the order-$2$ (degree-$4$) moment matrix $M_2(y)$ is block diagonal except for a single $2$-by-$2$ block coupling $ V_2$ and $V_3$.

\subsection{Orbit parameters}

Invariance and homogeneity force the degree-$4$ moments to be determined by five parameters, which correspond to the orbits of monomials under the group:
\begin{align*}
	&a := y_{0004}, \quad
	b := y_{4000} = y_{0400} = y_{0040}, \quad
	c := y_{2200} = y_{0220} = y_{2020}, \\
	&d := y_{2002} = y_{0202} = y_{0022}, \quad
	e := y_{1111}.
\end{align*}
All other moments are zero.
Let $O_{CL} : \R^5 \to \R^{70}$ be the linear map constructing the pseudo-moment vector $y$ from parameters $(a,b,c,d,e)$.

\subsection{Certificate spectrahedron}

If the moment matrix $M_2(y)$ is constructed in the
monomial order
\[
\{x_1^2,\ x_1x_2,\ x_1x_3,\ x_1x_4,\ x_2^2,\ x_2x_3,\ x_2x_4,\ x_3^2,\ x_3x_4,\ x_4^2\},
\]
let $Q\in\R^{10\times 10}$ be the orthogonal matrix whose {columns} are the new
orthonormal basis vectors written in the group invariant coordinates, ordered as
\[
\{\tfrac{1}{\sqrt3}(x_1^2+x_2^2+x_3^2),\ x_4^2,\ \tfrac{1}{\sqrt2}(x_1^2-x_2^2),\ \tfrac{1}{\sqrt6}(x_1^2+x_2^2-2x_3^2),\ x_1x_2,\ x_3x_4,\ x_2x_3,\ x_1x_4,\ x_1x_3,\ x_2x_4\}.
\]
Explicitly,
\[
Q=
\begin{bmatrix}
	\frac{1}{\sqrt3}&0&\frac{1}{\sqrt2}&\frac{1}{\sqrt6}&0&0&0&0&0&0\\
	0&0&0&0&1&0&0&0&0&0\\
	0&0&0&0&0&0&0&0&1&0\\
	0&0&0&0&0&0&0&1&0&0\\
	\frac{1}{\sqrt3}&0&-\frac{1}{\sqrt2}&\frac{1}{\sqrt6}&0&0&0&0&0&0\\
	0&0&0&0&0&0&1&0&0&0\\
	0&0&0&0&0&0&0&0&0&1\\
	\frac{1}{\sqrt3}&0&0&-\frac{2}{\sqrt6}&0&0&0&0&0&0\\
	0&0&0&0&0&1&0&0&0&0\\
	0&1&0&0&0&0&0&0&0&0
\end{bmatrix}.
\]
Then the moment matrix can be block diagonalized
\[
Q^T\,M_2(y)\,Q = \diag\left(M_{21}(y),\ b-c,\ b-c,\ M_{22}(y),\ M_{22}(y),\  M_{22}(y)\right)  
\]
with
\[
M_{21}(y)= \begin{bmatrix} b+2c & \sqrt3\,d\\ \sqrt3\,d & a \end{bmatrix}, \quad
M_{22}(y) = \begin{bmatrix} c & e\\ e & d \end{bmatrix}.
\]
Therefore $M_2(y) \succeq 0$ if and only if $a \geq 0$, $b \geq 0$, $c \geq 0$, $d \geq 0$, $b \geq c$, 
$(b+2c)\,a\ \ge\ 3d^2$, $cd\ \ge\ e^2$.
Note that if $\ell_y(p_{CL})=a+3c-4e < 0$ then $e \geq 0$. Therefore, the set of valid pseudo-moment certificates for $p_{CL}$ is the convex quadratic cone
\[
K_{CL}
:=O_{CL}(\{(a,b,c,d,e)\in\mathbb R^5_+:\ 
a+3c-4e<0,\ b\geq c,\ (b+2c)a\ge3d^2,\ cd\ge e^2
\}).
\]

\subsection{Exact certificates}

Let us construct rational points in \(K_{CL}\).

\begin{algorithm}\label{alg:kcl}
\emph{Step 0.}
Choose any rationals $c>0$, $f>0$, $e \geq (3c+f)/4$  and let $a = 4e-3c-f$.

\emph{Step 1.}
Choose any rational $g\ge 0$ and let 
$d = e^2/c + g$.

\emph{Step 2.}
Choose any rational $b \geq \max(c,\ 3d^2/a-2c)$.
\end{algorithm}

\begin{proposition}
Algorithm \ref{alg:kcl} generates a rational vector in $K_{CL}$.
\end{proposition}

\begin{proof}
Initially we have $a=4e-3c-f>0$ and $a+3c-4e=-f<0$ (strict separation). Step 1 gives
$cd = c (e^2/c+g) = e^2 + cg  \ge\ e^2$,
so $cd \geq e^2$ and $d\ge 0$. In Step 2 we explicitly enforce $b\ge c$, and
$(b+2c)a \ \ge\ (3d^2/a-2c+2c )a = 3d^2$,
so the quadratic inequality $(b+2c)a \geq 3d^2$ holds.
\end{proof}

Let us generate an integer certificate.
At Step 1 let $c=2$, $f=1$, $e=2\geq (3\cdot 2+1)/4=7/4$ and $a=4e-3c-f=1$. At Step 2 choose $g=1$ to get $d=3$. At Step 3, let $b=24 \geq \max(2,3\cdot 9-4)=23$. The resulting vector
\[
(a,b,c,d,e) = (1,24,2,3,2)
\]
corresponds to an interior point of $\Sigma^*_{4,4}$, i.e. $\rank M_2(y)=10$ is maximal.

\subsection{Extreme rays}

\begin{proposition}
	If $y \in K_{CL}$ then  \(\rank M_2(y) \in \{6,7,9,10\}\).
\end{proposition}

\begin{proof}
	First observe that if $(a,b,c,d,e)\in K_{CL}$, then $b>c$. Indeed, assume $b=c$. From $(b+2c)a\ge 3d^2$ we get $3ca\ge 3d^2$, hence $d^2\le ca$ and
	$\sqrt{cd} \le\ \sqrt{c\,\sqrt{ca}}\;=\;c^{3/4}a^{1/4}$. By the weighted arithmetic-geometric inequality  with weights $(1,3)$, it holds $(a+3c)/4 \ge(a\,c^3)^{1/4}\;=\;c^{3/4}a^{1/4}$.
	Therefore $\sqrt{cd}\le (a+3c)/4$, which contradicts the strict separation $a+3c-4e<0$ together with $e\le \sqrt{cd}$. Hence $b\ne c$, and since $b\ge c$ we must have $b>c$.
	
	It follows that $\rank M_2(y) = 2 + r_1 + 3r_2$ upon defining  $r_1:=\rank M_{21}(y)$ and $r_2:=\rank M_{22}(y)$. Note that
	\[
	r_1=\begin{cases}
		1 &\Longleftrightarrow\ (b+2c)a=3d^2,\\
		2 &\Longleftrightarrow\ (b+2c)a>3d^2,
	\end{cases}
	\qquad
	r_2=\begin{cases}
		1 &\Longleftrightarrow\ e^2=cd,\\
		2 &\Longleftrightarrow\ e^2<cd.
	\end{cases}
	\]
	and hence 
	\[
	\begin{aligned}
		\rank M_2(y)=10 &\iff (b{+}2c)a>3d^2\ \ \text{and}\ \ e^2<cd,\\
		\rank M_2(y)=9  &\iff (b{+}2c)a=3d^2\ \ \text{and}\ \ e^2<cd,\\
		\rank M_2(y)=7  &\iff (b{+}2c)a>3d^2\ \ \text{and}\ \ e^2=cd,\\
		\rank M_2(y)=6  &\iff (b{+}2c)a=3d^2\ \ \text{and}\ \ e^2=cd.
	\end{aligned}
	\]
\end{proof}

The pseudo-moment certificates $y \in K_{CL}$ that correspond to extreme rays of $\Sigma^*_{4,4}$ are characterized by a moment matrix $M_2(y)$ of rank 6, see Theorem \ref{ranks}.

\begin{algorithm}\label{alg:KCL-extreme}
\emph{Step 0.} Pick rationals \(u,v>0\) with \(4v>3u\).

\emph{Step 1.} Let $c:=u^2$, $d:=v^2$, $e:=uv$.

\emph{Step 2.} Choose any rational $b \ge \max(c,\ \frac{3v^4}{u(4v-3u)}-2u^2)$.

\emph{Step 3.} Let $a\ :=\ \frac{3d^2}{\,b+2c\,}.$
\end{algorithm}

\begin{proposition}\label{prop:KCL-extreme}
	Algorithm~\ref{alg:KCL-extreme} generates a rational extreme ray of \(y\in K_{CL}\).
\end{proposition}

For an illustration, take \(u=v=1\) (so \(4v>3u\) holds). Step~1 gives \((c,d,e)=(1,1,1)\).
The threshold in Step~2 is \(\max\{1,\ 3/(1\cdot 1)-2\}=1\).
Pick \(b=2\) and set \(a=3\cdot 1^4/(2+2)=\tfrac34\).
Then
\[
(a,b,c,d,e)=\Big(\tfrac34,\ 2,\ 1,\ 1,\ 1\Big).
\]

We report below a few integer examples: 
\[
\begin{array}{ccccc|c}
a & b & c & d & e & \rank \\\hline
1 & 24 & 2 & 3 & 2 & 10 \\ 
1 & 23 & 2 & 3 & 2 & 9 \\
2 & 8 & 3 & 3 & 3 & 7 \\
4 & 11 & 1 & 4 & 2 & 7 \\
1 & 8 & 2 & 2 & 2 & 6 \\
3 & 14 & 1 & 4 & 2 & 6 \\
4 & 10 & 1 & 4 & 2 & 6
\end{array}
\]

The above rank 9 certificate
\[
(a,b,c,d,e)=(1,23,2,3,2)
\]
can be decomposed as a convex combination of two rank 6 extreme ray certificates. 
The construction keeps $(a,d)$ fixed and moves $(b,c,e)$ on the two rank $6$ boundary curves so that the mean of $(b,c,e)$ matches $(23,2,2)$.
Both endpoints achieve the same strict separation value $a+3c-4e=-1$, so
they lie strictly inside $K_{CL}$ while remaining rank 6.
The two endpoints are
	\[
	\begin{aligned}
		 (\,1,\ 23-\tfrac{8\sqrt2}{3},\ 2+\tfrac{4\sqrt2}{3},\ 3,\ 2+\sqrt2),\ \text{ weight }1/2 \\[2pt]
		 (\,1,\ 23+\tfrac{8\sqrt2}{3},\ 2-\tfrac{4\sqrt2}{3},\ 3,\ 2-\sqrt2 ),\ \text{ weight }1/2.
	\end{aligned}
\]
Similarly the rank 7 certificate
\[
(a,b,c,d,e)=(4,11,1,4,2)
\]
can be decomposed into two rank 6 extreme rays
	\[
\begin{aligned}
	(\tfrac{8}{3},\,16,\,1,\,4,\,2),\ \text{ weight }3/8 \\[2pt]
	(\tfrac{24}{5},\,8,\,1,\,4,\,2),\ \text{ weight }5/8.
\end{aligned}
\]

Computing systematically these decompositions using rational arithmetic, or numerically stable floating point arithmetic, seems to be an interesting research direction. It would be a natural extension of the extraction algorithm described in \cite{HL05}, see also \cite[Section 4.3]{L10}.

\section{Conclusion}

A pseudo-moment certificate is a proof that a given positive polynomial is not SOS. The proof is based on convex duality, it is a hyperplane separating the cones of SOS polynomials and positive polynomials. 

In this note we describe how to exploit the symmetries of a polynomial to construct exact pseudo-moment certificates in low-dimensional cases. For a polynomial invariant under a group of transformations, the search for a certificate can be restricted to linear functionals that are also invariant under that group. This restriction reduces the complexity of the problem: instead of a large set of pseudo-moments, the functional is defined by a small number of orbit parameters. The condition that the moment matrix must be positive semidefinite, when expressed in terms of these few parameters, describes a low-dimensional spectrahedral cone. The structure of this spectrahedron is often sufficiently elementary that one can find a rational point satisfying the required positivity and negativity conditions analytically or by inspection, thereby yielding an exact, verifiable proof without the need for numerical solvers.

An independent third party can verify this certificate by performing two simple checks on the provided sequence of pseudo-moments. First, they compute the value of the associated linear functional on the polynomial, and verify that it is strictly negative. Second, they check that the functional is non-negative on the entire cone of SOS polynomials. This infinite-dimensional condition is tractably and exactly verified by constructing the finite-dimensional moment matrix from the pseudo-moments and confirming that it is positive semidefinite. If the certificate is provided with integer or rational pseudo-moments, the resulting moment matrix has rational entries. Its positive semidefiniteness can then be certified using exact methods, such as checking that all principal minors are non-negative, which involves only determinant calculations that can be performed without error in integer arithmetic. This transforms the verification into a finite sequence of exact algebraic computations, yielding an irrefutable proof.

Interestingly, the structure of a pseudo-moment certificate can be much simpler for some forms than for others, a phenomenon directly linked to the size of their symmetry groups. While living in the same cone of positive ternary sextics, the Robinson form is invariant under the full symmetric group, whereas the Motzkin form possesses a smaller symmetry group. Generally speaking, a larger symmetry group imposes more constraints on an invariant linear functional, reducing the number of independent orbit parameters needed to define the pseudo-moments. This reduction simplifies the problem by describing the set of valid certificates as a spectrahedral cone in a much lower-dimensional space. The resulting spectrahedron is not only simpler to analyze but also makes the task of finding an exact rational certificate analytically more tractable, as exemplified by the relative simplicity of the certificate for the highly symmetric Robinson form compared to that of the Motzkin form.
Note however that an excess of symmetry can so strongly constrain the structure of a polynomial that the gap between positivity and being SOS vanishes entirely. For example, it was proven that every positive ternary octic form that is also fully symmetric (i.e., invariant under all permutations of its variables) and even must be SOS \cite{H99,R00}. Consequently, for this highly symmetric class of polynomials, no pseudo-moment certificate can be constructed as there is nothing to separate from the SOS cone.

A significant advantage of constructing pseudo-moment certificates analytically is the ability to exert fine control over their algebraic properties, most notably the rank of the resulting moment matrix. By strategically choosing the orbit parameters to satisfy certain algebraic relations -- such as forcing specific vectors into the kernel of the matrix -- one can intentionally construct a certificate whose moment matrix is rank-deficient. This is geometrically significant, as a pseudo-moment vector corresponds to an extreme ray of the pseudo-moment cone if and only if its moment matrix has some specific rank. The possible ranks of these extreme rays are known in well-studied low-dimensional cases \cite{BS17}, providing a concrete goal for the analytical construction and a deeper understanding of the facial structure of the cone.

\section*{Acknowledgments}

This work initiated from exchanges with Heng Yang and Shucheng Kang on extreme rays of pseudo-moment cones. It also benefited from discussions with Tobias Metzlaff.

\end{document}